\documentclass[12pt,twoside]{article}

\usepackage{amsmath,amssymb,amsthm}
\usepackage{multirow,graphicx,color,subcaption,wrapfig,multicol}
\usepackage{graphicx,color,subcaption}

\newtheorem{theorem}{Theorem}[section]

\theoremstyle{definition}
\newtheorem{definition}[theorem]{Definition}

\theoremstyle{remark}

\numberwithin{equation}{section}

\theoremstyle{plain}
\newtheorem{thm}{Theorem}[section]
\newtheorem{prop}{Proposition}[section]
\newtheorem{lem}{Lemma}[section]

\theoremstyle{remark}
\newtheorem{rem}{Remark}[section]
\newtheorem{exmp}{Example}[section]

\def\P{{\mathbb {P}}}                       
\def\E{{\mathbb {E}}}

\def\I{{\mathbb {I}}}

\def\PD{{\mathcal P}}

\def\VD{{\mathcal V}} 
\def\PD{{\mathcal P}} 
\def\WD{{\mathcal W}}

\def\00{{\boldsymbol {0}}}
\def\11{{\boldsymbol {1}}}

\begin{document}
\setcounter{page}{1}

\pagestyle{myheadings}
\markboth{ B. Fazekas, I. Fazekas}{Convergence of sequences of ordered selections}

\centerline{\large{\bf {Convergence of sequences of ordered selections}}}

\bigskip

\centerline{ {B. Fazekas, I. Fazekas}}

\bigskip
University of Debrecen,  P.O. Box 400, 4002 Debrecen, Hungary

\medskip \bigskip

\begin{abstract}
In this paper, we introduce a convergence notion for ordered selections.
Our convergence notion is based on subpermutation densities and convergences of the marginal distributions.
A particular case of this convergence is the well-known convergence of permutation sequences.
We also introduce a family of probability measures called generalized permutons.
We show that in the family of generalized permutons several convergence notions are equivalent.
We embed the set of ordered selections to the set of generalized permutons.
We prove that any convergent sequence of ordered selections has a limit which is a generalized permuton.
Moreover, any generalized permuton is the limit of a sequence of ordered selections.
Our results are generalizations of well-known theorems on convergence of permutation sequences to permutons.
\end{abstract}

\renewcommand{\thefootnote}{}
{\footnotetext{
{\bf E-mail:} fazekas.istvan@inf.unideb.hu}}



\section{Introduction} \label{introduction}
\setcounter{equation}{0}
In \cite{Hoppen}, the limiting behaviour of sequences of permutations was studied.
If a sequence of permutations $\sigma_{k} =(\sigma_k(1), \sigma_k(2), \dots , \sigma_k(k))$
is convergent in the sense of permutation densities as $k\to \infty$, 
then it has a limit object.
The limit object is a probability measure on the unit square such that both marginals are uniform 
distributions on the unit interval.
In \cite{Hoppen}, the limit object is called limit permutation, but its usual name is copula
(see \cite{Grubel}) and nowadays it is often called permuton (see \cite{Glebov}).

In \cite{Hoppen}, the study of limits of permutation sequences was inspired by results on limits of graph sequences,
see e.g. \cite{LoSze}.
Since the publication of the groundbreaking paper \cite{Hoppen},
a rich theory of asymptotic behaviour of permutation sequences has been developed, see 
\cite{Bassino}, \cite{Alon} and the references therein.

In this paper, we extend the results of \cite{Hoppen} to ordered selections.
We select $m$ elements out of the set $\{1,2,\dots , n\}$ without replacements and 
assume that the order of the selection matters.
Denote the selected numbers by $\nu= \nu_{n,m} =(\nu(1), \nu(2), \dots , \nu(m))$.
We call $\nu_{n,m}$ an $(n,m)$-permutation.
We shall study the limiting behaviour of $\nu_{n,m}$ as $n,m \to \infty$ so that 
$\frac{m}{n}  \to \lambda \in (0,1]$.
In \cite{Hoppen}, it was shown that the permutation densities determine the convergence of a
sequence of usual permutations.  
However, we can see that the permutation densities alone do not describe the asymptotic stability
of a sequence of $(n,m)$-permutations.
To this end, we need also the distribution of the selected `places' in
the sequence $\{1,2,\dots , n\}$.
Using together the permutation densities and the distribution of the selected `places',
we can introduce a convergence notion which we denote by $\nu_{n,m} \xrightarrow[]{t} $.
Our goal is to describe the limit object.
Because $n \to \infty$, the limit object is outside of the set of permutations.
Following the ideas of \cite{Hoppen}, we can identify the limit object $\mu$ as a 
probability measure which we call $\lambda$-permuton.

The probability measure $\mu$ is called a  $\lambda$-permuton, 
if for a certain $\lambda$, $0<\lambda\le 1$,
$\mu$ is concentrated on $[0, \lambda] \times [0,1]$,
its first marginal distribution function is uniform on $[0, \lambda]$, 
and the second marginal distribution function is continuous and its slope is at most $\lambda$.
For a fixed $\lambda$, let $\WD_{\lambda}$ denote the set of all $\lambda$-permutons.
Let $\WD = \cup_{\lambda\in (0,1]} \WD_{\lambda}$.
The elements of $\WD$ are called generalized permutons.

In the set $\WD$ of generalized permutons, we introduce the notion of convergence in the sense of 
permutation densities which we denote by $\mu_n \xrightarrow[]{t} \mu$.
We show that for generalized permutons, this convergence is equivalent to the usual convergence of
 probability measures and as well as to the convergence in the sense of rectangular distance 
 and also in the sense of $L_\infty$ distance of the distribution functions.

The  $(n,m)$-permutations can be embedded into the space of generalized permutons.
In this way, we can prove that 
any convergent  sequence of $(n,m)$-permutations has a limit, which is a generalized permuton. 
Without going into the details of the precise definitions, we can present our main result.

\begin{thm}[\bf{Main theorem}]  \label{main}
Let $\nu_{n}=\nu_{n,m}$ be a convergent sequence of $(n,m)$-permutations: 
$\nu_{n} \xrightarrow[]{t} $.
Let $\mu_{n}$ be the generalized permuton corresponding to $\nu_{n}$, $n=1,2, \dots$.
Then the sequence $\mu_n$ is convergent, i.e. there exists a generalized permuton 
$\mu$ such that $\mu_{n} \xrightarrow[]{t} \mu$.
\end{thm}

So we can consider the generalized permuton $\mu$ as the limit of the sequence of the $(n,m)$-permutations,
so we can write $\nu_{n} \xrightarrow[]{t} \mu$.

We can see, that for $\lambda=1$ a $\lambda$-permuton is a usual permuton.
Moreover, for $n=m$, an $(n,m)$-permutation is a usual permutation.
So our result contains the original result of \cite{Hoppen}, Theorem 1.6 (i) as a particular case.
For the proof of our theorem, we use appropriate modifications of the methods of \cite{Hoppen}.

Our next theorem shows that any generalized permuton
is the limit of a sequence of $(n,m)$-permutations.

\begin{thm} \label{invMain}
Let $\mu$ be a generalized permuton.
Then there exists a sequence $\nu_{n}=\nu_{n,m_n}$ of $(n,m_n)$-permutations,
such that $\nu_{n} \xrightarrow[]{t} \mu$ as $n\to\infty$.
\end{thm}

This theorem is a generalization of Theorem 1.6 (ii) of \cite{Hoppen}.
To prove our result, we use directly  Theorem 1.6 (ii) of \cite{Hoppen} and some  geometric transformations.

In Section \ref{variations}, the definition and  main properties of the generalized permutons are presented.
In Section \ref{distances}, the distances of generalized permutons are studied.
In Section \ref{convergences}, convergence theorems for generalized permutons are proved.

\section{$(n,m)$-permutations and $\lambda$-permutons} \label{variations}
\setcounter{equation}{0}
Let $[n]= \{1,2,\dots , n\}$ be the set of the first $n$ positive integers.
Let $m$ be a fixed positive integer, assume $m\le n$.
We select $m$ elements out of the set $[n]$ without replacements and 
assume that the order of the selection matters.
Denote the selected numbers by $\nu= \nu_{n,m} =(\nu(1), \nu(2), \dots , \nu(m))$.
We call $\nu$ an $(n,m)$-permutation.
The set of $(n,m)$-permutations with fixed $n$ and $m$ is denoted by $\VD_{n,m}$ 
and the set of all $(n,m)$-permutations by $\VD= \cup_{n,m}\VD_{n,m}$.
The elements of $\VD$ are called generalized permutations.
We shall study the limiting behaviour of $\nu= \nu_{n,m}$ as $n,m \to \infty$ so that 
$\frac{m}{n} \to \lambda$, where $\lambda \in (0,1]$ is fixed.

Later we shall apply the notion of the permutation density.
To this end it will be convenient to distinguish the usual and the generalized permutations.
We shall denote the set of all permutations 
$\sigma= \sigma_{k} =(\sigma(1), \sigma(2), \dots , \sigma(k))$
of $[k]$ by $\PD_{k}$ and the set of all permutations
by $\PD= \cup_{k=1}^\infty \PD_{k}$.

We associate a probability measure $\mu =\mu_{n,m}$ to the $(n,m)$-permutation $\nu$ as follows.
We divide the rectangle $[0, \frac{m}{n}] \times [0,1]$ into $nm$ squares of sizes $\frac{1}{n}\times\frac{1}{n}$ 
and the weights of the $m$ squares picked by $\nu$ will be defined as $\frac{1}{m}$,
the weights of the other squares will be $0$.
Before going to the precise definition, we show a simple example.

\begin{exmp} \label{ex1}
Let $n=5$, $m=3$ and $\nu= (2,4,1)$.
Then the measure $\mu$ associated to $\nu$ is visualized on Figure \ref{fig1}.
\begin{figure}[h!]
    \centering
        \includegraphics[width=0.6\textwidth, height=0.6\textwidth]{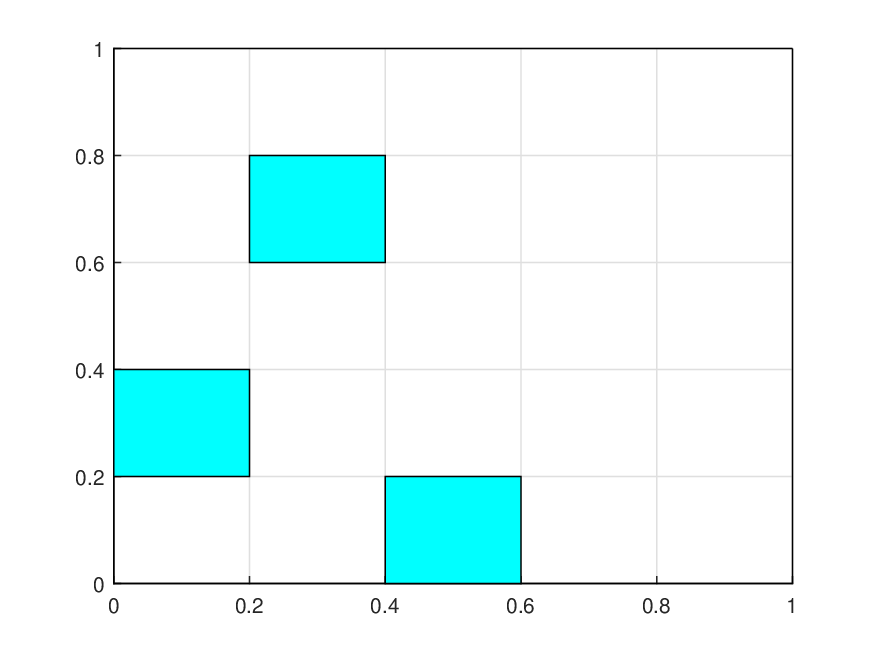}
    \caption{The measure $\mu$ in Example \ref{ex1}.}\label{fig1}
\end{figure}

\noindent
The marginal densities are
$$
f_x(x) =\begin{cases} \frac{5}{3},& \text{if} \  x\in [0, \frac{3}{5}) , \\
                      0, & \text{otherwise}, 
                      \end{cases}
$$
$$
f_y(y) =\begin{cases} \frac{5}{3},& \text{if} \  x\in [0, \frac{2}{5}) \cup [\frac{3}{5}, \frac{4}{5}) , \\
                      0, & \text{otherwise}. 
                      \end{cases}
$$
The marginal distribution functions are
$$
F_x(x) =\begin{cases} 0, & \text{if} \ x\le 0, \\
                       \frac{5}{3}x,& \text{if} \  x\in (0, \frac{3}{5}] , \\
                      1, & \text{if} \ x> \frac{3}{5}, 
                      \end{cases}
$$
$$
F_y(y) =\begin{cases} 0, & \text{if} \ x\le 0, \\
                      \frac{5}{3}x,& \text{if} \  x\in (0, \frac{2}{5}] , \\
                      \frac{2}{3},& \text{if} \  x\in (\frac{2}{5}, \frac{3}{5}] , \\
                      \frac{5}{3}x -\frac{1}{3},& \text{if} \  x\in (\frac{3}{5}, \frac{4}{5}] , \\
                      1, & \text{if} \ x> \frac{4}{5}. 
                      \end{cases}
$$
The marginal densities and distribution functions of the measure $\mu$ are visualized on 
Figure \ref{fig2}.
\begin{figure}[h!]
    \centering
        \includegraphics[width=0.95\textwidth, height=0.7\textwidth]{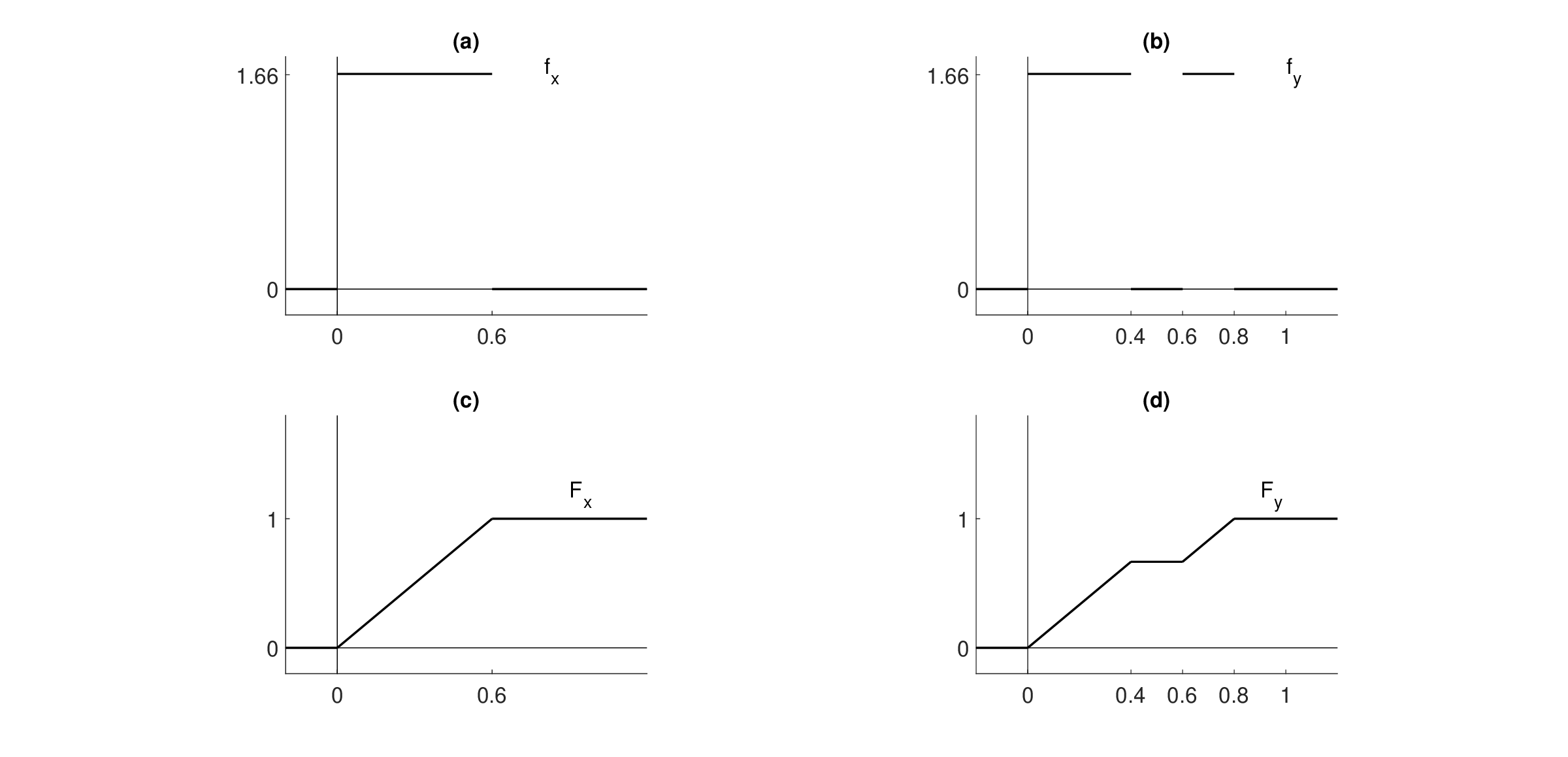}
    \caption{The marginal densities and distribution functions of the measure 
    $\mu$ in Example \ref{ex1}.}\label{fig2}
\end{figure}
\end{exmp}

\begin{definition} \label{W}
Let $m,n$ be fixed positive integers with $m\le n$.
Let $\nu= \nu_{n,m} =(\nu(1), \nu(2), \dots , \nu(m))$ be an $(n,m)$-permutation.
We denote by $S_{h,l}$ the square
\begin{equation}  \label{Shl}
S_{h,l}= \left[\frac{h-1}{n}, \frac{h}{n} \right) \times  \left[\frac{l-1}{n}, \frac{l}{n} \right), \ \
h=1,\dots,m, \ l=1,\dots,n.
\end{equation}
Let 
\begin{equation}  \label{fnu}
f_\nu (x,y) = \frac{n^2}{m}, \ \ \text{if} \ \ (x,y)\in  S_{h,\nu(h)} \ \ \text{for some} \ \ 
h\in \{ 1, \dots , m \}
\end{equation}
and let $f_\nu (x,y) =0$ otherwise.
Then $f_\nu$ is a probability density function of a probability measure $\mu= \mu_\nu$ 
defined on the rectangle $[0, \frac{m}{n})\times [0, 1)$.
$\mu_\nu$ is called the generalized permuton corresponding to the $(n,m)$-permutation $\nu$.
The set of the generalized permutons corresponding to the $(n,m)$-permutations 
from the set $\VD_{n,m}$ is denoted by $\WD_{n,m}$.
\end{definition}

By Definition \ref{W}, any $(n,m)$-permutation $\nu \in \VD_{n,m}$ determines a unique 
generalized permuton $\mu \in \WD_{n,m}$.
We can see that this relation is one-to-one, i.e.
any generalized permuton $\mu \in \WD_{n,m}$ determines a unique $(n,m)$-permutation $\nu \in \VD_{n,m}$.

\begin{rem}
Let $\lambda = \frac{m}{n}$.
The marginal distribution of the generalized permuton $\mu \in \WD_{n,m}$ 
on the $x$-axis is uniform with density function 
$$
f_x(x) =\begin{cases} \frac{1}{\lambda},  & \text{if} \  x\in [0, \lambda) , \\
                      0, & \text{otherwise}, 
                      \end{cases}
$$
and with distribution function 
$$
F_x(x) =\begin{cases} 0, & \text{if} \ x\le 0, \\
                       \frac{1}{\lambda}x,& \text{if} \  x\in (0, \lambda] , \\
                      1, & \text{if} \ x> \lambda  .
                      \end{cases}
$$
The marginal density function $f_y$ on the $y$-axis is a step function with values $\frac{1}{\lambda}$ or
$0$, 
while the marginal distribution function $F_y$ on the $y$-axis is a broken line
with slopes $\frac{1}{\lambda}$ or $0$.
\end{rem}

\begin{definition}
Let $\nu \in \VD_{n,m}$ and let $\mu \in \WD_{n,m}$ be the unique generalized 
permuton corresponding to it.
Then the marginal distribution functions $F_x$ and $F_y$ of $\mu$ are considered as the marginal
distribution functions of $\nu$.
\end{definition}

Our aim is to find the limit of the $(n,m)$-permutation sequence $\nu_{n,m}$ as $n,m \to \infty$ so that 
$\frac{m}{n} \to \lambda \in (0,1]$.
We can suppose that the number $m$ is determined by $n$, i.e. $m=m_n$, but we omit the subscript $n$.

If $\mu_{n,m}$ denotes the generalized permuton corresponding to $\nu_{n,m}$,
then the usual convergence of the probability measures $\mu_{n,m} \Rightarrow  \mu$
implies the convergence of the marginals,
so the first marginal distribution function of $\mu$ should be
$F_x(x) =\frac{1}{\lambda}x$, if $x\in [0, \lambda]$, 
and the second marginal distribution should be concentrated to $[0,1]$ and it should be 
a Lipschitz function with Lipschitz constant $\frac{1}{\lambda}$, i.e. its increments should satisfy
$F_y(b) -F_y(a) \le \frac{b-a}{\lambda}$ for $0\le a < b \le 1$.
These considerations lead to the following general definition of $\lambda$-permutons.

\begin{definition}
The probability measure $\mu$ is called a $\lambda$-permuton, 
if for a fixed $\lambda$, $0<\lambda\le 1$,
 $\mu$ is concentrated on $[0, \lambda] \times [0,1]$,
its first marginal distribution function is
$F_x(x) =\frac{1}{\lambda}x $ if $x\in [0, \lambda]$, 
and the second marginal distribution function is continuous and it
 satisfies
$F_y(b) -F_y(a) \le \frac{b-a}{\lambda}$ for $0\le a < b \le 1$.
For a fixed $\lambda$, let $\WD_{\lambda}$ denote the set of all $\lambda$-permutons.
Let $\WD = \cup_{\lambda\in (0,1]} \WD_{\lambda}$.
The elements of $\WD$ are called generalized permutons.
\end{definition}

We shall see that $\WD$ contains the limits of sequences of $(n,m)$-permutations.

\begin{rem}
We shall study the convergence of sequences on generalized 
permutons $\mu_{n}$ with parameters
$\lambda_n$ such that $\lambda_n  \to \lambda$ for some fixed $\lambda \in (0,1]$.
So we usually suppose that $0< \underline{\lambda} \le \lambda_n \le \overline{\lambda} \le 1$ 
for all $n$.
Sometimes it is convenient to complement the domain of $\mu_{n}$ with a zero measure stripe, 
and then we can suppose that each $\mu_{n}$ is defined on $[0, 1] \times [0,1]$. 
   \end{rem}  

We recall that $x_q$, $0< q<1$, is the $q$-quantile of the one-dimensional continuous
distribution function $F$ if $x_q = \sup \{x \ : \ F(x) = q\}$.  

Now, we define the probability measure $\mu_\sigma$ corresponding to 
the $\lambda$-permuton $\mu$ and a permutation $\sigma$.

\begin{definition}  \label{def3}
Let $\mu$ be a $\lambda$-permuton with marginal distribution functions 
$F_x$, $F_y$ and parameter $\lambda$.
Let $k$ be a fixed positive integer.
We define the $k$-subdivision of $\mu$ as the following array of rectangles $R_{i,j}$.
Let
$$
0=x_0 < x_1 < \dots < x_k=\lambda, \qquad
0=y_0 < y_1 < \dots < y_k=1
$$
so that $x_i$ is the $\frac{i}{k}$-quantile of $F_x$ ($i=1,\dots ,k-1$)
and
$y_j$ is the $\frac{j}{k}$-quantile of $F_y$ ($j=1,\dots ,k-1$).
Then define the rectangle $R_{i,j}$ as
$$
R_{i,j} = [x_{i-1}, x_i ) \times [y_{j-1}, y_j ), \quad i,j= 1, 2, \dots , k.
$$
Let $\sigma=(\sigma(1), \sigma(2), \dots , \sigma(k))\in \PD_k$ be a permutation of $[k]$.
Define the two-dimensional density function $f_\sigma$ as
\begin{equation}
f_\sigma (x,y)= \frac{1}{k} \cdot \frac{1}{x_i- x_{i-1}} \cdot \frac{1}{y_{\sigma(i)}- y_{\sigma(i)-1}},
\ \ \text{if} \ \ (x,y) \in R_{i,\sigma(i)}
\end{equation}
for some $i= 1, \dots , k$
and let $f_\sigma (x,y)=0$ otherwise.
The corresponding distribution function is denoted by $ F_\sigma (x,y)$.
The density function $f_\sigma$ defines a probability measure $\mu_\sigma$ on  
$[0, \lambda] \times [0,1]$.
\end{definition}

According to $\mu_\sigma$, the measure of any $R_{i,\sigma(i)}$ is $\frac{1}{k}$, 
and the measure of any other rectangle of the $k$-subdivision is $0$.

\begin{exmp} \label{ex2}
Consider an example for $\mu_\sigma$ in Definition \ref{def3}.
Let $k=4$ and $\sigma=(2,3,1,4)$.
The four rectangles having positive, i.e. $\frac{1}{4}$ measure, are visualized on
Figure \ref{fig3}.
\begin{figure}[h!]
    \centering
        \includegraphics[width=0.8\textwidth, height=0.6\textwidth]{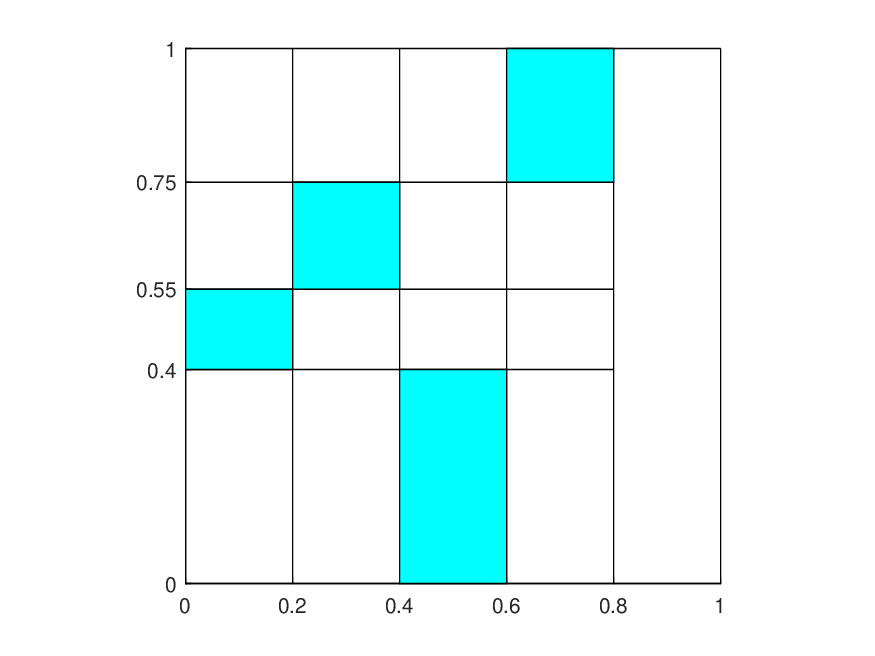}
    \caption{The measure $\mu_\sigma$ in Example \ref{ex2}.}\label{fig3}
\end{figure}
Here the quantiles are determined by the marginal distribution functions of  $\mu$.
In our example $\lambda=0.8$ and the quantiles on the $x$ axis are $0.2, 0.4, 0.6$
and on the $y$ axis they are $0.4, 0.55, 0.75$.
\end{exmp}

\begin{definition}  \label{def4}
Let $\mu$ be a generalized permuton.
Choose a sample from the distribution $\mu$, i.e.
let $(\xi_1, \eta_1), \dots , (\xi_k, \eta_k)$
be independent identically distributed random vectors such that the distribution
of $(\xi_i, \eta_i)$ is $\mu$ for each $i\in [k]$.
Let
$$
\xi_1^\ast< \xi_2^\ast < \dots < \xi_k^\ast , \qquad
\eta_1^\ast< \eta_2^\ast < \dots < \eta_k^\ast
$$
be the ordered samples
(remark that the marginal distribution functions are continuous, so the sample elements are
distinct with $\mu$-probability $1$, so we can use strict inequalities in the above relations).
The  $\mu$-random permutation is defined as
$$
\sigma=(\sigma(1), \sigma(2), \dots , \sigma(k)),
$$
if for any $i\in[k]$ we have $\xi_i^\ast =\xi_l$ and $\eta_l = \eta_{\sigma(i)}^\ast$ for 
a certain $l$.
We shall denote the above $\mu$-random permutation by $\mu^{(k)}$.
Now, the $\mu$-random generalized permuton is defined as
$\sigma(k,\mu) =\mu_\sigma$, where $\sigma$ is a $\mu$-random permutation 
and $\mu_\sigma$ is from Definition \ref{def3}.
\end{definition} 

That is, $\sigma(k,\mu)$ is a generalized permuton depending on the randomness,
as now $\sigma$ in the expression $\mu_\sigma$ is based on the sample from distribution $\mu$.

\begin{rem}
Let $F$ be a two-dimensional distribution function with marginal distribution functions
$F_x$  and $F_y$.
Then for $x_1<x_2$ and $y_1<y_2$ we have
\begin{equation} \label{H7}
F(x_2, y_2) -  F(x_1, y_1) \le (F_x(x_2) -  F_x(x_1)) + (F_y(y_2) -  F_y(y_1)).
\end{equation}
If the slopes of $F_x$  and $F_y$ are bounded by $C$, then
for $x_1<x_2$ and $y_1<y_2$ we have
\begin{equation} \label{H7+}
F(x_2, y_2) -  F(x_1, y_1) \le C((x_2 -  x_1) + (y_2 -  y_1)).
\end{equation}
\end{rem}

\begin{rem}
Assume that the maximum of the slopes of the marginal distribution functions of $\mu$
is $\frac{1}{\lambda}$.
As the $\frac{i}{k}$ quantiles of the marginals of $\mu_\sigma$ are the same as those of $\mu$,
the slopes of both marginal distribution functions of $\mu_\sigma$
are also at most $\frac{1}{\lambda}$.
\end{rem}
%

\section{Distances of generalized permutons} \label{distances}
\setcounter{equation}{0}
%

\begin{definition}  \label{def4+}
Let $\mu_i$ be a generalized permuton and let $F_i$ be the two-dimensional distribution function 
corresponding to it, $i=1,2$.
The $d_\infty$  distance of the generalized permutons $\mu_1$ and $\mu_2$ is defined as
\begin{equation} \label{H33}
d_\infty(\mu_1,\mu_2)
=\|F_1-F_2 \|_\infty =
\sup_{x,y\in[0,1]} |F_1(x,y) -F_2(x,y)|.
\end{equation}
The $d_\square$  distance is defined as
\begin{equation} \label{H32}
d_\square (\mu_1,\mu_2)
=d_\square(F_1, F_2 )=
\sup_{\substack{x_1, x_2, y_1, y_2 \in[0,1] \\ x_1 <x_2, y_1< y_2}} 
|\mu_1([x_1, x_2] \times [ y_1, y_2]) - \mu_2([x_1, x_2] \times [ y_1, y_2])|.
\end{equation}
\end{definition}

Then, see \cite{Hoppen}, we have
\begin{equation} \label{H34}
d_\infty(\mu_1,\mu_2) \le d_\square (\mu_1,\mu_2) \le 4 d_\infty(\mu_1,\mu_2) .
\end{equation}

\begin{lem}  \label{H4.2}
There exists a $k_0$ such that for $k> k_0$ and for any $\lambda$-permuton $\mu$ we have
\begin{equation} \label{H37}
\P \left( d_\square \left(\mu, \sigma(k, \mu)\right) \le 16 k^{-\frac{1}{4}}   \right)
\ge 1- \frac{1}{2} e^{-\sqrt{k}},
\end{equation}
where $\sigma(k, \mu)$ is an arbitrary $\mu$-random $\lambda$-permuton introduced in Definition \ref{def4}.
\end{lem}

\begin{proof}
Let $\mu$ be a $\lambda$-permuton on $[0,\lambda] \times [0,1]$ and let 
$F$ be its two-dimensional distribution function.
Let $\sigma(k, \mu)$ be a $\mu$-random $\lambda$-permuton and let $F_k$ be its 
two-dimensional distribution function.
By \eqref{H34}, if we prove
\begin{equation} \label{H39}
\P \left( \sup_{x,y\in[0,1]} |F(x,y) -F_k(x,y)| \ge 4 k^{-\frac{1}{4}}   \right)
\le \frac{1}{2} e^{-\sqrt{k}},
\end{equation}
then we obtain \eqref{H37}.

Let us use the notation from Definition \ref{def3}.
Let $(x,y)\in R_{i,j}$.
Then
\begin{eqnarray} \label{H18}
&&| F_k(x,y)- F(x,y) |\le \\
&\le & | F_k(x,y)- F_k(x_i,y_j) |+ | F_k(x_i,y_j)- F(x_i,y_j) | +| F(x_i,y_j)- F(x,y) |\le  \nonumber \\
&\le & \frac{2}{k}+ | F_k(x_i,y_j)- F(x_i,y_j) | + \frac{2}{k}, \nonumber
\end{eqnarray}
using \eqref{H7} and the fact that $x_i$ is the $\frac{i}{k}$ quantile of $F_x$ and $y_j$ is the $\frac{j}{k}$ quantile of $F_y$.

Let $\varepsilon= k^{-\frac{1}{4}}$.
Let $(\xi_1, \eta_1), \dots , (\xi_k, \eta_k)$ be the sample from the distribution $\mu$.
Then we have
\begin{equation} \label{H41-}
F_k(x_i,y_j) = \frac{1}{k} \sum_{t=1}^k \I\{ \xi_t \le \xi_i^\ast, \eta_t \le \eta_j^\ast \} ,
\end{equation} 
where $\I$ denotes the indicator function.
Now, we shall show that for all $i,j\in [k]$
\begin{equation} \label{H41}
\P \left( F_k(x_i,y_j) > F(x_i,y_j) +3\varepsilon   \right)
< 3 e^{-2\sqrt{k}} .
\end{equation} 
Let
$x'_i$ and $y'_j$ be such that
\begin{equation*}
F_x(x_i+ x'_i) = \frac{i}{k }  + \varepsilon, \qquad  F_y(y_j+ y'_j) = \frac{j}{k }  + \varepsilon .
\end{equation*} 
So \eqref{H7} gives
\begin{equation} \label{H42}
F(x_i+ x'_i, y_j+y'_j) +\varepsilon \le F(x_i, y_j) + 3 \varepsilon .
\end{equation} 
Let $A$ be the following event
\begin{equation} \label{H42+}
A=  \left\{\frac{1}{k} \sum_{t=1}^k \I\{ \xi_t \le \xi_i^\ast, \eta_t \le \eta_j^\ast \} 
> F(x_i+ x'_i, y_j+y'_j) +\varepsilon  \right\} .
\end{equation} 
Then \eqref{H42} gives that the left-hand side of \eqref{H41} is not greater than $\P(A)$.
So we have to find an upper bound of $\P(A)$.
We have
\begin{eqnarray} \label{H43}
\P(A) &=& \P\left(A \cap \{ \xi^\ast_i <  x_i+ x'_i  \} \cap \{ \eta^\ast_j <  y_j+ y'_j  \} \right) \\
&+&\P\left(A \cap \left(\{ \xi^\ast_i \ge  x_i+ x'_i  \} 
\cup \{ \eta^\ast_j \ge  y_j+ y'_j  \} \right) \right) \nonumber \\
&\le&
\P\left(A \cap \{ \xi^\ast_i <  x_i+ x'_i  \} \cap \{ \eta^\ast_j <  y_j+ y'_j  \} \right) \nonumber \\
&+&\P\left(\xi^\ast_i \ge  x_i+ x'_i  \right) +
\P\left( \eta^\ast_j \ge  y_j+ y'_j  \} \right) . \nonumber 
\end{eqnarray}

We need the following large deviation result for the binomial distribution.
Let $\zeta$ be a binomial random variable with parameters $k$ and $p$.
Then for $\varepsilon>0$
\begin{equation} \label{H44}
\P\left(\frac{1}{k} \zeta \ge p+\varepsilon \right)
\le \exp (-2k \varepsilon^2 ), \ \
\P\left(\frac{1}{k} \zeta \le p-\varepsilon \right)
\le \exp (-2k \varepsilon^2 ) .
\end{equation}
Then, with $p= F_x(x_i+ x'_i) = \frac{i}{k }  + \varepsilon$, we have
\begin{eqnarray*}
&&\P\left(\xi^\ast_i \ge  x_i+ x'_i  \right) \\
&&=\P \left( \frac{1}{k} \sum_{t=1}^k \I\{ \xi_t \le x_i+ x'_i\}  \le \frac{i}{k} \right)
= \P\left(\frac{1}{k} \zeta \le p-\varepsilon \right)
\le \exp (-2k \varepsilon^2 ) .
\end{eqnarray*}
Similarly 
$$
\P\left(\eta^\ast_i \ge  y_j+ y'_j  \right) 
\le \exp (-2k \varepsilon^2 ) .
$$
Moreover,
\begin{eqnarray*} 
&& \P\left(A \cap \{ \xi^\ast_i <  x_i+ x'_i \} \cap \{ \eta^\ast_j <  y_j+ y'_j \} \right)  \\
 = &&
\P \Bigg( \left\{ \frac{1}{k} \sum\nolimits_{t=1}^k \I\{ \xi_t \le \xi_i^\ast, \eta_t \le \eta_j^\ast \} 
> F(x_i+ x'_i, y_j+y'_j) +\varepsilon  \right\}     \\
&& \cap \{ \xi^\ast_i <  x_i+ x'_i  \} \cap \{ \eta^\ast_j <  y_j+ y'_j  \} \Bigg)  \\
=&&
\P \Bigg( \left\{ \frac{1}{k} \sum\nolimits_{t=1}^k 
\I\{ \xi_t \le x_i+ x'_i, \eta_t \le y_j+ y'_j \} 
\ge  F(x_i+ x'_i, y_j+y'_j) +\varepsilon  \right\}  \Bigg)   \\
= &&  \P\left(\frac{1}{k} \zeta \ge p+\varepsilon \right)
\le \exp (-2k \varepsilon^2 ) .
\end{eqnarray*}
From the above inequalities
$$
\P(A) \le 3 \exp (-2k \varepsilon^2 ) = 3 \exp (-2k (k^{-\frac{1}{4}})^2 ) = 3 \exp (-2\sqrt{k}).
$$
%
Now, we have to prove that for all $i,j\in [k]$
\begin{equation} \label{H41+}
\P \left( F_k(x_i,y_j) < F(x_i,y_j) -3\varepsilon   \right)
< 3 e^{-2\sqrt{k}} .
\end{equation} 
Let now
$x''_i$ and $''_j$ be such that
\begin{equation*}
F_x(x_i- x''_i) = \frac{i}{k }  - \varepsilon, \qquad  F_y(y_j- y''_j) = \frac{j}{k }  - \varepsilon .
\end{equation*} 
So \eqref{H7} gives
\begin{equation} \label{H42++}
F(x_i- x''_i, y_j-y''_j) -\varepsilon \ge F(x_i, y_j) - 3 \varepsilon .
\end{equation} 
Let $A''$ be the following event
\begin{equation} \label{H42+++}
A''=  \left\{\frac{1}{k} \sum_{t=1}^k \I\{ \xi_t \le \xi_i^\ast, \eta_t \le \eta_j^\ast \} 
< F(x_i- x''_i, y_j-y''_j) -\varepsilon  \right\} .
\end{equation} 
Then \eqref{H42++} gives that the left-hand side of \eqref{H41+} is not greater than $\P(A'')$.
So we have to find an upper bound of $\P(A'')$.
We have
\begin{eqnarray} \label{H43+}
\P(A'') &=& \P\left(A'' \cap \{ \xi^\ast_i >  x_i- x''_i  \} \cap 
\{ \eta^\ast_j >  y_j- y''_j  \} \right) \\
&+&\P\left(A'' \cap \left(\{ \xi^\ast_i \le  x_i- x''_i  \} 
\cup \{ \eta^\ast_j \le  y_j- y''_j  \} \right) \right) \nonumber \\
&\le&
\P\left(A \cap \{ \xi^\ast_i >  x_i- x''_i  \} \cap \{ \eta^\ast_j >  y_j- y''_j  \} \right) \nonumber \\
&+&\P\left(\xi^\ast_i \le  x_i- x''_i  \right) +
\P\left( \eta^\ast_j \le  y_j- y''_j  \} \right) . \nonumber 
\end{eqnarray}
We again use the large deviation result for the binomial distribution.
Then
\begin{eqnarray*}
&&\P\left(\xi^\ast_i \le  x_i- x''_i  \right) \\
&&=\P \left( \frac{1}{k} \sum_{t=1}^k \I\{ \xi_t \le x_i- x''_i\}  \ge \frac{i}{k} \right)
= \P\left(\frac{1}{k} \zeta \ge p+\varepsilon \right)
\le \exp (-2k \varepsilon^2 ) .
\end{eqnarray*}
Similarly 
$$
\P\left(\eta^\ast_i \le  y_j- y''_j  \right) 
\le \exp (-2k \varepsilon^2 ) .
$$
Then
\begin{eqnarray*} 
&& \P\left(A'' \cap \{ \xi^\ast_i >  x_i- x''_i \} \cap \{ \eta^\ast_j >  y_j- y''_j \} \right)  \\
 = &&
\P \Bigg( \left\{ \frac{1}{k} \sum\nolimits_{t=1}^k \I\{ \xi_t \le \xi_i^\ast, \eta_t \le \eta_j^\ast \} 
< F(x_i- x''_i, y_j-y''_j) -\varepsilon  \right\}     \\
&& \cap \{ \xi^\ast_i >  x_i- x''_i  \} \cap \{ \eta^\ast_j >  y_j- y''_j  \} \Bigg)  \\
=&&
\P \Bigg( \left\{ \frac{1}{k} \sum\nolimits_{t=1}^k 
\I\{ \xi_t \le x_i- x''_i, \eta_t \le y_j- y''_j \} 
\le  F(x_i- x''_i, y_j-y''_j) -\varepsilon  \right\}     \\
= &&   \P\left(\frac{1}{k} \zeta \le p-\varepsilon \right)
\le \exp (-2k \varepsilon^2 ) .
\end{eqnarray*}
From the above inequalities
$$
\P(A'') \le 3 \exp (-2k \varepsilon^2 ) = 3 \exp (-2\sqrt{k}).
$$
By \eqref{H41} and \eqref{H41+}, we have
\begin{equation} \label{H40}
\P \left( |F_k(x_i,y_j) - F(x_i,y_j)| > 3\varepsilon   \right)
\le 6 e^{-2\sqrt{k}} .
\end{equation} 
Let $k$ be so large that $\frac{4}{k} <k^{-\frac{1}{4}}$.
Then, by \eqref{H18},
$$
| F_k(x,y)- F(x,y) |\le   k^{-\frac{1}{4}}+ | F_k(x_i,y_j)- F(x_i,y_j) | 
$$
if $(x,y) \in R_{i,j}$.
Then, for large enough $k$
\begin{eqnarray*} 
&&\P \left( \sup_{x,y\in[0,1]} |F(x,y) -F_k(x,y)| \ge 4 k^{-\frac{1}{4}}   \right) \\
&\le & 
\P \left( \max_{i,j\in[k]} |F(x_i,y_j) -F_k(x_i,y_j)| > 3 k^{-\frac{1}{4}}   \right) \\
&\le & 
\sum_{i,j\in[k]}
\P \left( |F(x_i,y_j) -F_k(x_i,y_j)| > 3 k^{-\frac{1}{4}}   \right) \\
&\le & 
6 k^2 e^{-2\sqrt{k}}
= \left( 12 k^2 e^{-\sqrt{k}} \right) \left(\frac{1}{2} e^{-\sqrt{k}}\right) \le
\frac{1}{2} e^{-\sqrt{k}}.     
\end{eqnarray*}
This completes the proof.
\end{proof}
%

\begin{definition}  \label{def5}
Let $\mu$ be a generalized permuton.
Recall that the $\mu$-random  permutation of order $k$ is denoted by
$\mu^{(k)}$, see Definition \ref{def4}.
Let $\tau\in \PD_k$ be a non-random  permutation of order $k$.
The density of $\tau$  in $\mu$ is defined as
\begin{equation}  \label{H24}
t(\tau, \mu) =\P \left(\tau= \mu^{(k)}\right).
\end{equation}
\end{definition}

\begin{rem}
A generalized permuton $\mu$ uniquely determines its marginal distributions $F_x$ and $F_y$
(where $F_x$ is uniform on $[0, \lambda]$, with $0< \lambda\le 1$).
Moreover, $\mu$ determines  $t(\tau, \mu)$ for any $\tau \in \PD$.
\end{rem}

\begin{prop}  \label{HLemma5.1}
The marginal distribution functions $F_x$ and $F_y$ and the densities
$t(\tau, \mu)$, $\tau \in \PD$, uniquely determine the generalized permuton $\mu$.
\end{prop}

\begin{proof}
Fix a positive integer $k$ and an elementary event $\omega$.
Then the marginal distribution functions $F_x$ and $F_y$ and the 
$\mu$-random permutation $\sigma = \mu^{(k)} =\mu^{(k)}(\omega)$ 
determine the two-dimensional distribution function $F_{\sigma}(x,y)$ like in 
Definition \ref{def3}.
As $|F_{\sigma}| \le 1$, the expectation $\E F_{\sigma}$ exists and it is uniquely determined by
the permutation densities $t(\tau, \mu)$, $\tau \in \PD_k$.
By \eqref{H39}, the two-dimensional distribution function $F$ belonging to $\mu$
is approximated by $\E F_{\sigma}$ with precision $\varepsilon>0$ if $k> k_\varepsilon$ is large enough.
So $F$ is unique if $F_x$ and $F_y$ and the densities
$t(\tau, \mu)$, $\tau \in \PD$, are given.
\end{proof}
%

\section{Convergence of generalized permutons} \label{convergences}
\setcounter{equation}{0}
%

\begin{lem}  \label{lem41}
Let $\mu$  and $\mu_n$, $n=1,2, \dots$ be  generalized permutons.
Then the following statements are equivalent.
\begin{equation} \label{dinf}
d_\infty(\mu_n,\mu) \to 0 \quad \text{ as} \quad n\to \infty,
\end{equation}
\begin{equation} \label{dsquare}
d_\square(\mu_n,\mu) \to 0 \quad \text{ as} \quad n\to \infty,
\end{equation}
\begin{equation} \label{dmeasure}
\mu_n \Rightarrow \mu \quad \text{ as} \quad n\to \infty.
\end{equation}
\end{lem}

\begin{proof}
The equivalence of \eqref{dinf} and \eqref{dsquare} is a consequence of the inequalities
\eqref{H34}.
The equivalence of \eqref{dinf} and \eqref{dmeasure} can be proved by repeating the argument in
the first part of the proof of Lemma 2.1 in \cite{Hoppen}.

We mention that we need that $\mu$ is non-degenerated, 
that is the measure $\mu$ is not concentrated on the vertical axis, only in the implications 
 \eqref{dmeasure} $\rightarrow$ \eqref{dinf} and
\eqref{dmeasure} $\rightarrow$ \eqref{dsquare}, as in those cases we need \eqref{H7+}.
\end{proof}

\begin{definition}  \label{deftconv}
Let $\mu$  and $\mu_n$, $n=1,2, \dots$ be  generalized permutons.
Let $F_x$ and $F_y$ be the marginal distribution functions of $\mu$.
Let $F^{(n)}_x$ and $F^{(n)}_y$ be the marginal distribution functions of $\mu_n$, $n=1,2, \dots$.
We say that the sequence $\mu_n$, $n=1,2, \dots$ converges to $\mu$  in the sense of permutation 
densities, if for any permutation $\tau\in \PD$ we have
$$
\lim_{n\to \infty} t(\tau, \mu_n) = t(\tau, \mu),
$$
and $F^{(n)}_x \Rightarrow F_x$, $F^{(n)}_y \Rightarrow F_y$ as $n\to \infty$.
For this convergence we shall use the notation $\mu_n \xrightarrow[]{t} \mu$.
\end{definition}

\begin{lem}  \label{H5.3}
Let $\mu$  and $\mu_n$, $n=1,2, \dots$ be  generalized permutons.
Then
$\mu_n \xrightarrow[]{t} \mu$ and $\mu_n \Rightarrow \mu$ are equivalent.
\end{lem}  

\begin{proof}
(i)
Let $\mu_n \Rightarrow \mu$.
Then for the marginal distribution functions we have $F^{(n)}_x \Rightarrow F_x$ and 
$F^{(n)}_y \Rightarrow F_y$ as $n\to \infty$.
Let $\tau\in \PD_k$ be a permutation of order $k$.
We are going to show that
\begin{equation}  \label{X}
\lim_{n\to \infty} t(\tau, \mu_n) = t(\tau, \mu).
\end{equation}
Let $(\xi_i, \eta_i)$, $i=1, \dots , k$, be a sample from the distribution $\mu$. 
For any fixed positive integer $n$,
let $(\xi^{(n)}_i, \eta^{(n)}_i)$, $i=1, \dots , k$, be a sample from the distribution $\mu_n$.
With probability $1$, the elements of each sample are different.
 $\mu_n \Rightarrow \mu$ implies that the distribution of the $2k$-dimensional vector
 $$
 (\xi^{(n)}_1, \eta^{(n)}_1,\dots , \xi^{(n)}_k, \eta^{(n)}_k)
 $$
 converges to that of
 $$
 (\xi_1, \eta_1,\dots , \xi_k, \eta_k).
 $$
 For any vector $(x_1, \dots ,x_k)$, we denote its ordered version by
 $x^{\ast}_1 < \dots < x^{\ast}_k$.
 
 Let $A_\tau\subseteq [0,1]^{2k}$ be the set of vectors
 $(x_1, y_1, \dots ,x_k, y_k)$ such that
 $x^{\ast}_i = x_l$ implies $y^{\ast}_{\tau(i)} = y_l$ for each $i=1,\dots , k$.
 Then 
 \begin{equation}  \label{H47}
 t(\tau, \mu_n) = 
 \P^{(n)} \left( \left(\xi^{(n)}_i, \eta^{(n)}_i \right)_{i=1}^{k} \in A_\tau \right) =
 \mu_n^{k} (A_\tau) ,
 \end{equation} 
 and 
\begin{equation}  \label{H47+}
 t(\tau, \mu) = 
 \P\left( \left(\xi_i, \eta_i \right)_{i=1}^{k} \in A_\tau \right) =
 \mu^{k} (A_\tau) ,
 \end{equation}  
 where $\mu^{k} $ and $\mu_n^{k} $ denotes the $k$-fold product measures.
 Now, we show that the boundary of $A_\tau$ is of $\P$-measure zero.
 To this end, observe that on the boundary of $A_\tau$,
 two coordinates must be equal, i.e. $x_i=x_j$ or $y_i=y_j$ for some $i\ne j$.
 But the marginals of $\mu$ are continuous, so it has probability $0$.
 So $\mu_n \Rightarrow \mu_n$ implies $\mu_n^{k} (A_\tau) \to \mu^{k} (A_\tau)$ as 
 $n \to \infty$.
 So \eqref{X} is true.
 
 (ii)
 To prove the other direction, we apply Theorem 3.3 of Chapter I of \cite{Bill}.
 I.e. for probability measures on a metric space we have:
 $\mu_n \Rightarrow \mu$ if and only if any subsequence $\mu_{n'}$ contains a further
 subsequence $\mu_{n''}$ such that $\mu_{n''} \Rightarrow \mu$.
In our case  $\mu_n$, $n=1, 2, \dots$, is tight, so any subsequence $\mu_{n'}$ contains a further
 subsequence $\mu_{n''}$ such that $\mu_{n''} \Rightarrow \tilde{\mu}$.
 By our assumption, the marginal distributions of $\mu_{n''}$ converge to the marginal distributions of
 $\mu$, but they converge also to those of $\tilde{\mu}$.
 It implies, that $\tilde{\mu}$ is a generalized permuton.
 Moreover, by the previous part of the theorem, 
 $\lim_{n\to \infty} t(\tau, \mu_{n''}) = t(\tau, \tilde{\mu})$.
 But, by the assumptions, 
 $\lim_{n\to \infty} t(\tau, \mu_{n''}) = t(\tau, {\mu})$.
 By Proposition \ref{HLemma5.1}, $\mu = \tilde{\mu}$.
 So $\mu_n \Rightarrow \mu$.
\end{proof}

\begin{definition}  \label{deftPermDens}
Let $\nu\in \VD_{n,m}$ be an $(n,m)$-permutation and let $\tau\in \PD_k$ be a permutation, $k\le m \le n$.
Let $\Gamma(\tau, \nu)$ denote the number of those vectors 
$(x_1, x_2, \dots , x_k) \in [m]^k$ for which $x_1 < x_2 < \dots < x_k$ and such that
$\nu(x_i) < \nu(x_j)$ if and only if $\tau(i) < \tau(j)$ for any $i,j$.
The density of $\tau$ in $\nu$ is defined as
$$
t(\tau, \nu) = \frac{\Gamma(\tau, \nu) }{ \binom{m}{k} }.
$$
\end{definition}

\begin{definition}  \label{deftconvNU}
Let $m_n \to \infty$ be a sequence of positive integers, $m_n\le n$ for any $n$
and assume that $\lim_{n\to\infty} \frac{m_n}{n} = \lambda >0$.
Let $\nu_n\in \VD_{n,m_n}$ be an $(n,m)$-permutation for each $n$.
Let $F^{(n)}_x$ and $F^{(n)}_y$ be the marginal distribution functions of $\nu_n$, $n=1,2, \dots$.
We say that the sequence $\nu_n$, $n=1,2, \dots$ is convergent, if for any permutation $\tau\in \PD$, 
the sequence
$t(\tau, \nu_n)$ is convergent, 
and the sequences of functions $F^{(n)}_x$ and $F^{(n)}_y$ are convergent as $n\to \infty$.
\end{definition}

\begin{lem}  \label{H3.5}
Let $\nu\in \VD_{n,m}$ be an $(n,m)$-permutation and let $\mu\in \WD_{n,m}$ be the 
generalized permuton corresponding to it.
Let $\tau\in \PD_{k}$ be a permutation, $k\le m$.
Then
$$
|t(\tau, \nu) - t(\tau, \mu) | \le \frac{1}{m} \binom{k}{2} .
$$
\end{lem} 

\begin{proof}
Let $\tau\in \PD_{k}$ be fixed.
Let $(\xi_i, \eta_i)$, $i=1, \dots , k$, be a sample from the distribution $\mu$.
Let $A$ be the event that $\xi^{\ast}_i = \xi_l$ and 
$\eta^{\ast}_{\tau(i)} = \eta_l$ for each $i=1,\dots , k$.
 Then 
$\P(A) = t(\tau, \mu)$.
Let $B$ be the event that any square $S_{h,l}$ from Definition \ref{W}
 contains at most one element $(\xi_i, \eta_i)$.
Then $\P(A|B) = t(\tau, \nu)$.
We also see that 
$\P(\bar{B}) \le \frac{1}{m} \binom{k}{2}$.
Using (28) of \cite{Hoppen},
$$
|t(\tau, \nu) - t(\tau, \mu) | = | \P(A)- \P(A|B)| \le \P(\bar{B}) \le \frac{1}{m} \binom{k}{2} .
$$
This completes the proof.
\end{proof}

\begin{prop}  \label{F}
Let $m_n \to \infty$ be a sequence of positive integers, $m_n\le n$ for any $n$
and assume that $\lim_{n\to\infty} \frac{m_n}{n} = \lambda >0$.
For each $n$, let $\nu_n\in \VD_{n,m_n}$ be an $(n,m_n)$-permutation and let $\mu_n\in \WD_{n,m_n}$
be the generalized permuton corresponding to it.
The sequence
$\nu_n$ is convergent if and only if 
$\mu_n \xrightarrow[]{t} \mu$ for some $\lambda$-permuton $\mu$.
The $\lambda$-permuton $\mu$ can be considered as the limit of the sequence of the $(n,m)$-permutations  $\nu_n$.
So we shall use the notation $\nu_n \xrightarrow[]{t} \mu$.
\end{prop}

\begin{proof}
The marginal distribution functions of $\nu_n$ are the same as those of $\mu_n$.
So the limiting behaviour of the marginal distribution functions are the same.
From Lemma \ref{H3.5}, we see that the limits of
$t(\tau, \nu_n)$ and  $t(\tau, \mu_n)$ are the same for any $\tau\in \PD_{k}$.
Now, we show that the limit $\mu$ exists.
By tightness, any subsequence $\mu_{n'} $ contains a further subsequence 
$\mu_{n''} $ converging to some $\lambda$-permuton $\mu$, i.e. $\mu_{n''} \Rightarrow \mu$.
The marginal distributions of $\mu$ and its permutation densities are uniquely determined by the
sequence $\nu_n$.
So, by Proposition \ref{HLemma5.1}, $\mu$ is uniquely determined and it does not depend on the
particular subsequence.
So $\mu_{n} \Rightarrow \mu$.
So, by Lemma \ref{H5.3}, $\mu_n \xrightarrow[]{t} \mu$.
\end{proof}

\begin{proof}[Proof of Theorem \ref{main}]
It is a simple consequence of Proposition \ref{F}
\end{proof}
%

%
\begin{proof}[Proof of Theorem \ref{invMain}]
Let $\mu$ be a $\lambda$-permuton, let $F$ be its distribution function and let $F_x$ and $F_y$ be the marginal
distribution functions.
As $\lambda \in (0, 1]$, we can find a sequence of positive integers $M_N$ with $M_N \le N$,  $N= 1, 2, \dots $
such that $\frac{M_N}{N} \to \lambda$ as $N \to \infty$.
Therefore, a simple argument shows, that we can approximate $\mu$ by $\frac{M_N}{N}$-permutons.
So, it is enough to prove our theorem for $\lambda= \frac{M}{N}$, where $M$ and $N$ are positive integers.

Let $x_i= \frac{i}{N}$ be the $\frac{i}{M}$-quantile of $F_x$, and let $y_i$ be the $\frac{i}{M}$-quantile of $F_y$, $i=0,1, \dots, M$.
We can choose $x_0=y_0= 0$, $x_M= \lambda$ and  $y_M= 1$.
Let
$$
R_{i,j} = (x_{i-1}, x_i] \times (y_{j-1}, y_j], \qquad i,j= 1,2,\dots , M.
$$
We call the rectangles
$$
G_{j} = (0, \lambda] \times (y_{j-1}, y_j] = \cup_{i=1}^M R_{i,j}, \qquad j= 1,2,\dots , M,
$$
gray stripes.
We see that $\mu(G_j) = \frac{1}{M}$, $j= 1,2,\dots , M$.

Now, we divide into parts the vertical axis using the points $0, \frac{1}{N}, \frac{2}{N}, \dots, 1$.
Choose the integers $l_1, \dots , l_M$ such that with notation
$\bar{Z}_j= \frac{l_j}{N}$ and $\underline{Z}_j= \frac{l_j-1}{N}$
the interval $( \underline{Z}_j , \bar{Z}_j]$ contains the quantile $y_j$ ($j= 1,2,\dots , M$).
As the slope of $F_y$ is at most $\frac{1}{\lambda} = \frac{N}{M}$, for different values of $j$, 
the intervals $( \underline{Z}_j , \bar{Z}_j]$ are disjoint.

Let us call the rectangles
$$
B_{j} = (0, \lambda] \times ( \underline{Z}_j , \bar{Z}_j] , \qquad j= 1,2,\dots , M,
$$
black stripes and the rectangles
$$
W_{j} = (0, \lambda] \times ( \bar{Z}_{j-1}, \underline{Z}_j] , \qquad j= 1,2,\dots , M,
$$
white stripes.
Some $W_j$ can be empty.
We see that the black stripes together with the white stripes offer a partition of $(0,\lambda]\times(0,1]$.
We shall approximate $\mu$ by another $\lambda$-permuton $\tilde{\mu}$ in the following way.
We shall transform the measure of any gray stripe $G_j$ to the corresponding black stripe $B_j$, such that
the measure of the white stripes become $0$.
We can do it by the following procedure.

Let the probability measure $\tilde{\mu}$ be uniformly distributed on the rectangle
$
(x_{i-1}, x_i] \times ( \underline{Z}_j , \bar{Z}_j] 
$
with
$$
\tilde{\mu} \left((x_{i-1}, x_i] \times ( \underline{Z}_j , \bar{Z}_j] \right) =
{\mu} \left((x_{i-1}, x_i] \times ( y_{j-1} , y_j] \right) 
$$
for any $i$ and $j$.
It defines $\tilde{\mu}$ on the black stripes.
Let the $\tilde{\mu}$ measure of any white stripe be $0$. 
Let $\tilde{F}$ be the distribution function of $\tilde{\mu}$, and let
$\tilde{F}_x$ and $\tilde{F}_y$ be the marginal distribution functions.
We see that $\tilde{F}_x$ is the uniform distribution on $(0,\lambda]$.
$\tilde{F}_y$ is the distribution function of a probability measure on $(0,1]$, it is
a continuous broken line, it is constant along any white stripe and it has slope $\frac{1}{\lambda} = \frac{N}{M}$ along any black stripe.
So $\tilde{\mu}$ is a $\lambda$-permuton.
We see that for the $\infty$-distance: $\| F - \tilde{F} \|_\infty  \le \frac{2}{M}$.

Now, we shall apply a transformation 
$$
L: (0, \lambda] \times ( 0,1] \to (0, \lambda] \times ( 0,\lambda],
$$
which is defined by the following procedure.
We remove all white stripes.
This way the black stripes will be shifted down and we obtain from $\tilde{\mu}$ 
a probability measure $\tilde{\mu}_\lambda$ on $(0, \lambda] \times ( 0,\lambda]$.
Both marginals of $\tilde{\mu}_\lambda$ will be uniform on $(0, \lambda] $.
Then, using the transformation $(x,y) \to (\frac{x}{\lambda}, \frac{y}{\lambda})$, we magnify 
the square $(0, \lambda] \times ( 0,\lambda]$ to the square $(0, 1] \times ( 0,1]$.
The image of the measure $\tilde{\mu}_\lambda$ will be the permuton $\tilde{\mu}_1$.
Now, applying the original result of \cite{Hoppen},
we can approximate $\tilde{\mu}_1$ by permutations $\sigma_k : [kM] \to [kM]$, $k=1,2,\dots$.
Let $\tilde{\sigma}_k$ denote the permuton corresponding to $\sigma_k $.
Then $d_\infty(\tilde{\mu}_1 , \tilde{\sigma}_k) \le \varepsilon_k$,
where $\varepsilon_k \to 0$ as $k\to \infty$.

Now, we perform backwards the previous transforms.
First, we reduce the size of the square $(0, 1] \times ( 0,1]$ by the transformation 
$(x,y) \to (\lambda x, \lambda y)$.
This transformation maps $\tilde{\sigma}_k$ to a probability $\beta_k$ on the square
$(0, \lambda] \times ( 0,\lambda]$ such that
$$
d_\infty(\tilde{\mu}_\lambda , \beta_k) = d_\infty(\tilde{\mu}_1 , \tilde{\sigma}_k) \le \varepsilon_k .
$$
The structure of $\beta_k$ is the following.
If we divide $(0, \lambda] \times ( 0,\lambda]$ into $(Mk)^2$ small squares of type 
$\left(\frac{i-1}{Nk}, \frac{i}{Nk} \right] \times \left(\frac{j-1}{Nk}, \frac{j}{Nk} \right]$,
then the $\beta_k$-measures of $Mk$ squares are $\frac{1}{Mk}$, and the 
$\beta_k$-measures of the remaining squares are $0$.
Moreover, the squares having $\frac{1}{Mk}$ measure represent a permutation of $[Mk]$.
Now, we perform the inverse of the transformation $L$.
That is, we insert back the white stripes between the black stripes so that the measure of any white stripe is 
defined as $0$.
This way the measure $\beta_k$ is transformed into a measure
$\tilde{\beta}_k \in \WD_{Nk, Mk}$.
That is $\tilde{\beta}_k $ is a $\lambda$-permuton which represents an $(Nk,Mk)$-permutation.
By the construction,
$d_\infty(\tilde{\mu}, \tilde{\beta}_k) = d_\infty(\tilde{\mu}_\lambda , \beta_k) \le \varepsilon_k $.
Therefore,
$$
d_\infty({\mu}, \tilde{\beta}_k) \le  
d_\infty(\mu, \tilde{\mu}) + d_\infty(\tilde{\mu}, \tilde{\beta}_k) \le \frac{2}{M} +\varepsilon_k ,
$$
which is arbitrarily small for appropriate $M$ and $k$.

\begin{figure}[h!] 
    \centering
    \begin{subfigure}[b]{0.45\textwidth}
        \includegraphics[width=\textwidth]{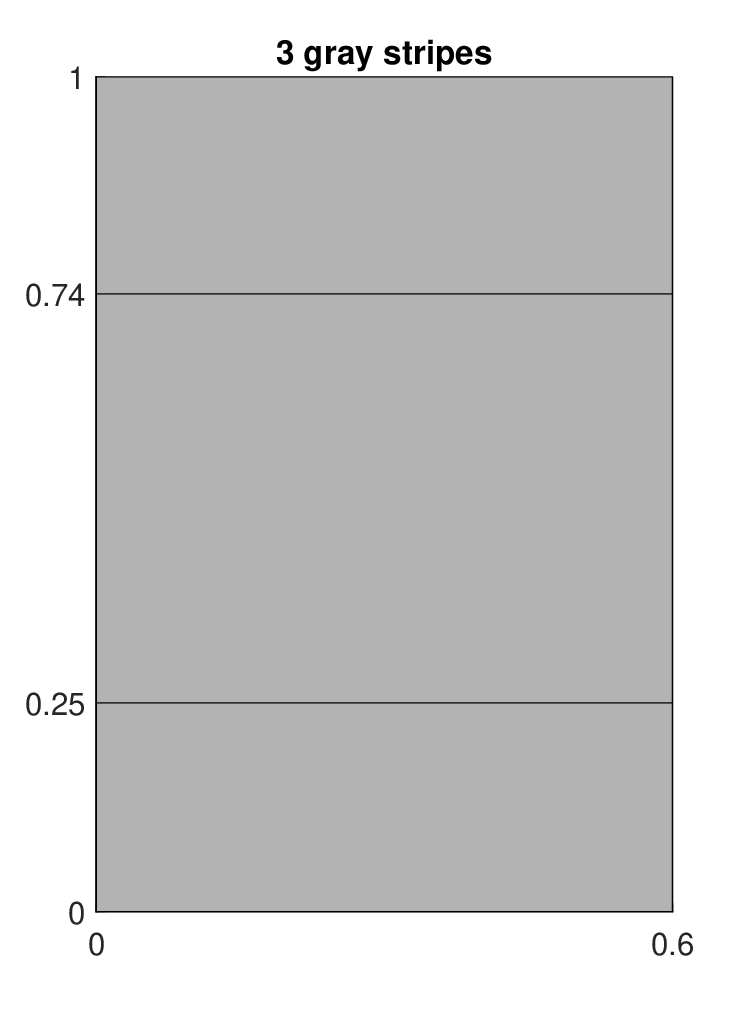}
        \caption{3 gray stripes}
        \label{fS1}
    \end{subfigure}
    \begin{subfigure}[b]{0.45\textwidth}
        \includegraphics[width=\textwidth]{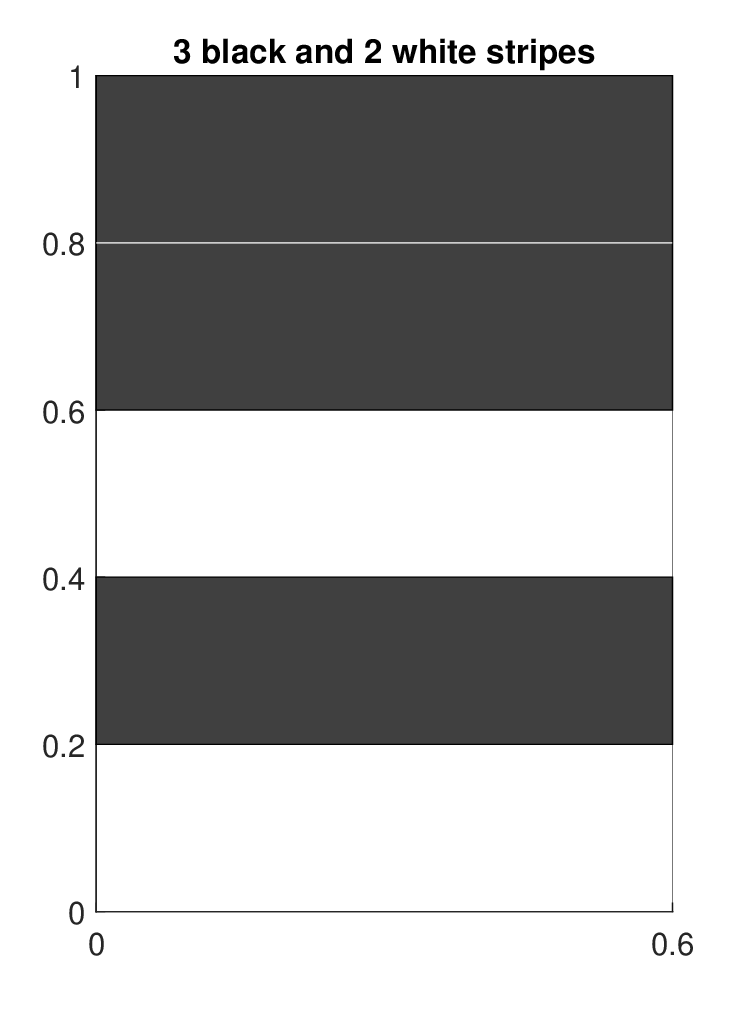}
        \caption{3 black and 2 white stripes}
        \label{fS2}
    \end{subfigure}
    \begin{subfigure}[b]{0.45\textwidth}
        \includegraphics[width=\textwidth]{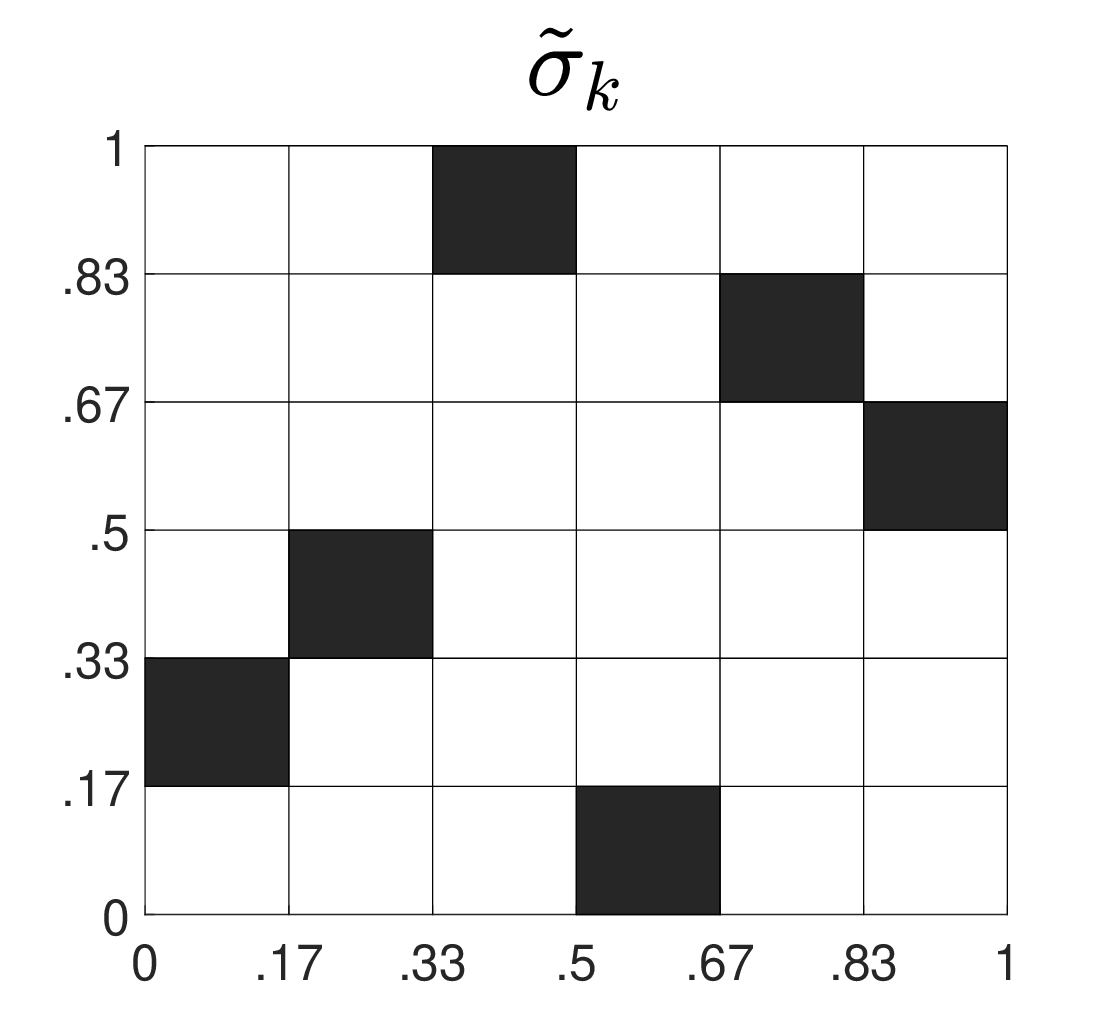}
        \caption{The permuton $\tilde{\sigma}_k$}
        \label{fS3}
    \end{subfigure}    
    \begin{subfigure}[b]{0.45\textwidth}
        \includegraphics[width=\textwidth]{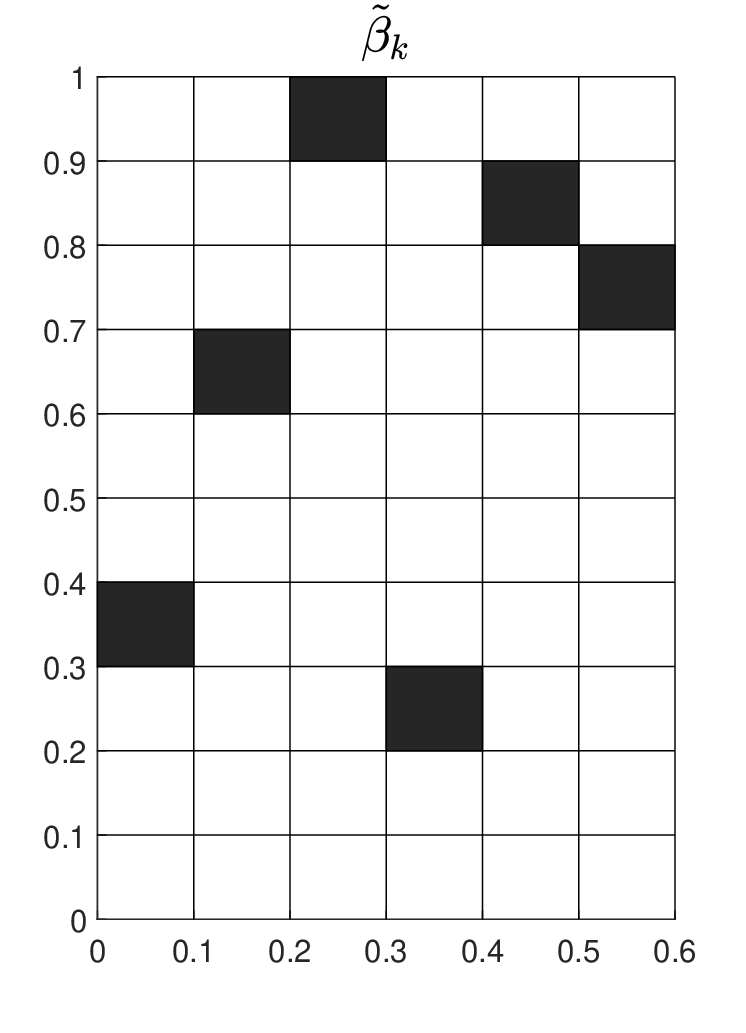}
        \caption{The generalized permuton $\tilde{\beta}_k$}
        \label{fS4}
    \end{subfigure}
    \caption{Four steps in the proof of Theorem \ref{invMain} for the  particular case given in Example \ref{exStripes}.} \label{figStripes}
\end{figure}

We see that the parameters of $\tilde{\beta}_k$ are $(Nk, Mk)$.
That is the sequence of the first parameter is $N, 2N, 3N,\dots$.
Here, we shall show, that the first parameter can run along all positive integers.
If $k$ is large enough, then we can interpolate between $\tilde{\beta}_k$ and $\tilde{\beta}_{k+1}$
by certain generalized permutons $\beta_{k,i} \in \WD_{Nk+i, Mk+j}$, $i=1,2,\dots , N-1$,
where the value of $j$ will be given by our next procedure.
Our starting measure is $\tilde{\beta}_{k+1} \in \WD_{Nk+N, Mk+M}$.
It can be visualized as a board containing $(Nk+N)(Mk+M)$ small squares arranged in 
$Nk+N$ rows and  $Mk+M$ columns.
$Mk+M$ small squares are black according to the corresponding
$(Nk+N,Mk+M)$-permutation
(to be black means that its measure is positive).
The remaining squares are white meaning that their measures are $0$.
Now, we remove $M$ black squares step by step.
When we remove a black square, then we delete its row and column, we distribute its measure
uniformly among the remaining black squares.
Finally, we transform the shape of the board 
to obtain again a generalized permuton.
We can see that the $d_\infty$ distance of the distribution function of the new generalized permuton and
the previous one is not greater than $\frac{3}{Mk}$.
As we should do this step $M$-times, the $d_\infty$ distances of the distribution functions of the obtained
 generalized permutons and the initial one is not greater than $\frac{3}{k}$.
After the above procedure, we remove step by step $N-M$ rows not containing black squares.
When we remove a row, then we reshape  the board to obtain again a generalized permuton.
We can see that the $d_\infty$ distance of the distribution function of the new generalized permuton and
the previous one is not greater that $\frac{2}{Mk}$.
As we should do this step $(N-M)$-times, the $d_\infty$ distances of the distribution functions of the obtained
 generalized permutons and the initial one is not greater than $\frac{2}{k}\left(\frac{N}{M} - 1 \right)$.
 Finally, if we add the previously obtained maximal distances, then we obtain 
 $$
 \frac{3}{k} +\frac{2}{k}\left(\frac{N}{M} - 1 \right) = \frac{1}{k}\left(\frac{2N}{M} +1 \right)
$$
which is arbitrarily small for $k$ being large enough.
\end{proof}

\begin{exmp}  \label{exStripes}
We present a simple example to visualize the proof of Theorem \ref{invMain}.
Let $N=5$ and $M=3$.
Assume that the $\frac{1}{3}$ quantile of $F_y$ is $0.25$ and its $\frac{2}{3}$ quantile is $0.74$.
So the three gray stripes are the stripes on subfigure \eqref{fS1} of Figure \ref{figStripes}.
As $N=5$, we divide the original rectangle into $5$ stripes.
So the three black and the two white stripes are shown on subfigure \eqref{fS2}.
The measure $\tilde{\mu}$ is concentrated on the black stripes.

Now, let $k=2$.
The usual (i.e. not generalized) permuton $\tilde{\sigma}_k$ is visualized on  subfigure \eqref{fS3}.
Then, we insert back the white stripes, so from $\tilde{\sigma}_k$ we obtain the generalized permuton
$\tilde{\beta}_k$.
The shape of $\tilde{\beta}_k$ is shown on subfigure \eqref{fS4}.
We can see that $\tilde{\beta}_k$ corresponds to the ordered selection 
$(4,7,10,3,9,8)$ out of the elements $\{1,2,3,4,5,6,7,8,9,10\}$.
\end{exmp}


\begin{thebibliography}{99}
\footnotesize

\bibitem{Alon}
{\sc Alon, N., Defant, C. and Kravitz, N.} (2022).
The runsort permuton.
{\em Advances in Applied Mathematics}, {\bf 139,} 102361. 
%

%
\bibitem{Bassino} 
{\sc Bassino, F., Bouvel, M., Féray, V., Gerin, L. and Pierrot, A.} (2018).
The Brownian limit of separable permutations.
{\em Ann. Probab.}
{\bf 46/4,} 2134--2189.
%

\bibitem{Bill} 
{\sc Billingsley, P.} (1968).
{\em Convergence of Probability Measures}, 
Wiley, New York-London-Sydney.
%

\bibitem{Glebov} 
{\sc Glebov, R., Grzesik, A., Klimo\v sov\'a, T. and  Kr\'al', D.} (2015).
Finitely forcible graphons and permutons.
{\em J. Combinatorial Theory, Ser. B},
{\bf 110,} 112--135.
%

\bibitem{Grubel} 
{\sc Gr\"ubel, R.} (2024).
Ranks, copulas, and permutons. 
{\em Metrika}, {\bf 87,} 155--182. 
%

\bibitem{Hoppen} 
{\sc Hoppen, C.,  Kohayakawa, Y.,  Moreira, C. G.,  R\'ath, B. and Sampaio, R. M.} (2013).
Limits of permutation sequences.
{\em J. Combinatorial Theory, Ser. B},
{\bf 103/1,} 93--113.
%

\bibitem{LoSze}
{\sc Lov\'asz, L. and Szegedy, B.} (2006).
Limits of dense graph sequences.
{\em Journal of Combinatorial Theory, Series B},
{\bf 96/6,} 933--957.

\end{thebibliography}
\end{document}